\documentclass[11pt]{article}

\usepackage[margin=1in]{geometry}
\usepackage{amsmath,amssymb,amsthm}
\usepackage{mathtools}
\usepackage{enumitem}
\usepackage{hyperref}
\hypersetup{hidelinks}
\newtheorem{theorem}{Theorem}[section]
\newtheorem{corollary}[theorem]{Corollary}
\newtheorem{lemma}[theorem]{Lemma}
\newtheorem{proposition}[theorem]{Proposition}
\theoremstyle{definition}
\newtheorem{definition}[theorem]{Definition}
\newtheorem{remark}[theorem]{Remark}

\renewcommand{\P}{\mathbb{P}}
\DeclareMathOperator{\dist}{dist}
\DeclareMathOperator{\Cov}{Cov}
\DeclareMathOperator{\Var}{Var}

\title{Finite-sample Borel--Cantelli inequalities under mixing conditions}
\author{Chatchawan Panraksa\\
Applied Mathematics Program, Mahidol University International College\\
999 Phutthamonthon 4 Road, Salaya, Nakhon Pathom 73170, Thailand\\
\texttt{chatchawan.pan@mahidol.ac.th}}
\date{}

\begin{document}
\maketitle

\begin{abstract}
We prove explicit one-lag, finite-$N$ lower bounds for
$\P\bigl(\bigcup_{k=1}^{N}A_{k}\bigr)$ that use only the marginal
probabilities $\P(A_{k})$ and a single selected-lag dependence
coefficient of the event-generated $\sigma$-fields. A residue-class
blocking argument gives, under $\varphi$-mixing, a bound with a free
spacing parameter $L\ge 0$, spacing constant $1/(L+1)$, and residual
governed by $\varphi(L+1)$; a strong-mixing covariance argument gives an
$\alpha$-mixing analogue with an additive residual
$\lceil N/(L+1)\rceil\,\alpha(L+1)$. When the lag-$(L+1)$ coefficient
vanishes, both reduce to the finite-sample $m$-dependent bound of
Panraksa (2026), and the spacing constant $1/(L+1)$ is sharp in this
zero-residual sense. A second-order Bonferroni refinement and a worked
geometrically $\varphi$-mixing example are included. The estimates are
non-asymptotic one-lag tools, complementary to second-moment and
variance criteria rather than competitors to them.
\end{abstract}



\section{Introduction}

The Borel--Cantelli lemma is a cornerstone of probability theory: for
events $(A_{k})_{k\ge1}$ on $(\Omega,\mathcal F,\P)$,
$\sum_{k}\P(A_{k})<\infty$ implies $\P(A_{k}\ \mathrm{i.o.})=0$, while
under independence $\sum_{k}\P(A_{k})=\infty$ implies
$\P(A_{k}\ \mathrm{i.o.})=1$ \cite{Durrett2010}. Because independence is
restrictive, a substantial literature relaxes it. Yoshihara
\cite{Yoshihara1979} and Tasche \cite{Tasche1997} obtained second
Borel--Cantelli conclusions under strong-mixing rates, and Dedecker,
Merlev\`ede and Rio \cite{DedeckerMerlevedeRio2022} developed
Borel--Cantelli and strong Borel--Cantelli criteria for weakly dependent
sequences with applications to Markov chains and dynamical systems.
These predecessors are asymptotic or variance-based.

For finite $N$ there is a complementary, non-asymptotic line. The
finite-sample $m$-dependent inequality of \cite{Panraksa2026} states
that, for an $m$-dependent family,
\begin{equation}\label{eq:mdep-baseline}
  \P\Bigl(\,\bigcup_{k=1}^{N}A_{k}\,\Bigr)
  \;\ge\; 1-\exp\!\Bigl(-\tfrac{1}{m+1}\,S_{N}\Bigr),
  \qquad S_{N}:=\sum_{k=1}^{N}\P(A_{k}),
\end{equation}
proved by splitting $\{1,\dots,N\}$ into $m+1$ residue classes modulo
$m+1$, applying independence within a class, and selecting the class of
maximal mass; Lu, Shi and Zhao \cite{LuShiZhao2026} recently obtained
related asymptotic refinements. In most time-series, stationary-process
and dynamical-systems applications, however, dependence \emph{decays}
with the gap rather than vanishing abruptly past a fixed range. It is
therefore natural to ask: to what extent does \eqref{eq:mdep-baseline}
persist when $m$-dependence is replaced by a quantitative mixing
condition?

This note answers that question. The main result
(Theorem~\ref{thm:phi}) shows that, for every integer $L\ge0$ and for
the uniform-mixing coefficient $\varphi$ of the event-generated
$\sigma$-fields,
\begin{equation}\label{eq:phi-intro}
  \P\Bigl(\,\bigcup_{k=1}^{N}A_{k}\,\Bigr)
  \;\ge\; 1-\exp\!\Biggl(-\frac{1}{L+1}
  \sum_{k=1}^{N}\bigl(\P(A_{k})-\varphi(L+1)\bigr)_{+}\Biggr),
\end{equation}
with a strong-mixing analogue (Theorem~\ref{thm:alpha}) carrying an
additive residual. The estimates are finite-$N$ and explicit: they use
only the marginals and the coefficient at the single chosen lag $L+1$,
with no limiting argument. Taking $L=m$, where $\varphi(m+1)=0$, recovers
\eqref{eq:mdep-baseline} exactly, and the spacing constant $1/(L+1)$ is
sharp in a precise zero-residual sense (Proposition~\ref{prop:sharp}).
A second-order Bonferroni refinement (Theorem~\ref{thm:second}) brings
local pairwise intersection probabilities into the exponent.

These bounds are one-lag tools. They are complementary to the
Chung--Erd\H{o}s second-moment bound \cite{ChungErdos1952}, which needs
all pairwise intersections, and to variance criteria using the full
mixing-rate profile \cite{ErdosRenyi1959,DedeckerMerlevedeRio2022}; they
are intended for the case in which only the marginals and a single
coefficient value are available. The present note is the mixing sequel
to the finite-sample $m$-dependent study \cite{Panraksa2026}.
Section~\ref{sec:prelim} fixes the mixing coefficients and three
elementary lemmas; Section~\ref{sec:main} proves the two inequalities
and the second-order refinement and establishes sharpness;
Section~\ref{sec:ex} gives a worked example and places the estimates in
context.

\section{Preliminaries}\label{sec:prelim}

Throughout, $(\Omega,\mathcal F,\P)$ is a probability space and
$(A_{k})$ are events. For integers $a,b\ge1$, put
$\mathcal F_{a}^{b}:=\sigma(A_{j}:a\le j\le b)$ and
$\mathcal F_{a}^{\infty}:=\sigma(A_{j}:j\ge a)$, with $\mathcal F_{a}^{b}$
trivial if $a>b$; we write $S_{N}:=\sum_{k=1}^{N}\P(A_{k})$,
$(x)_{+}:=\max(x,0)$, $\dist(I,J)=\inf\{|i-j|:i\in I,\,j\in J\}$, and
take an empty product to be $1$.

\begin{definition}[Mixing coefficients]\label{def:coeff}
For sub-$\sigma$-algebras $\mathcal A,\mathcal B\subseteq\mathcal F$,
\[
  \alpha(\mathcal A,\mathcal B):=\sup\bigl\{|\P(A\cap B)-\P(A)\P(B)|:
  A\in\mathcal A,\,B\in\mathcal B\bigr\},
\]
\[
  \varphi(\mathcal A,\mathcal B):=\sup\bigl\{|\P(B\mid A)-\P(B)|:
  A\in\mathcal A,\,\P(A)>0,\,B\in\mathcal B\bigr\}.
\]
For the event sequence, set the \emph{ambient} coefficients
$\alpha_{A}(n):=\sup_{k\ge1}\alpha(\mathcal F_{1}^{k},
\mathcal F_{k+n}^{\infty})$ and
$\varphi_{A}(n):=\sup_{k\ge1}\varphi(\mathcal F_{1}^{k},
\mathcal F_{k+n}^{\infty})$, and, for a fixed finite horizon $N$, the
\emph{finite restricted} coefficients
\[
  \alpha_{N}(n):=\!\!\sup_{1\le k\le N}\!\!\alpha(\mathcal F_{1}^{k},\mathcal F_{k+n}^{N}),
  \qquad
  \varphi_{N}(n):=\!\!\sup_{1\le k\le N}\!\!\varphi(\mathcal F_{1}^{k},\mathcal F_{k+n}^{N}),
\]
where $\mathcal F_{a}^{N}$ is understood to be trivial when $a>N$.
\end{definition}

Each finite statement below is read under exactly one convention --
ambient or finite restricted, not a mixture -- and we then write
$\alpha(n),\varphi(n)$ for the chosen reading. We use only two
elementary facts about these coefficients, both immediate from
Definition~\ref{def:coeff}: the monotonicity
$\alpha(\mathcal A_{0},\mathcal B_{0})\le\alpha(\mathcal A,\mathcal B)$
and $\varphi(\mathcal A_{0},\mathcal B_{0})\le\varphi(\mathcal A,\mathcal B)$
for $\mathcal A_{0}\subseteq\mathcal A$, $\mathcal B_{0}\subseteq\mathcal B$,
and the comparison $\alpha(\mathcal A,\mathcal B)\le
\varphi(\mathcal A,\mathcal B)$, valid because
$|\P(A\cap B)-\P(A)\P(B)|=\P(A)\,|\P(B\mid A)-\P(B)|$ when $\P(A)>0$
(so every $\varphi$-mixing sequence is $\alpha$-mixing); see
\cite{Bradley2005,Rio2017} for these standard normalisations. If the
$A_{k}$ are measurable with respect to an underlying process, the
process-level coefficients dominate the event-level ones by
monotonicity, so the hypotheses below may be verified at the process
level.

\begin{remark}[$m$-dependence as a limiting case]\label{rem:mdep}
A family is $m$-dependent if $\sigma(A_{i}:i\in I)$ and
$\sigma(A_{j}:j\in J)$ are independent whenever $\dist(I,J)>m$;
equivalently $\mathcal F_{1}^{k}$ and $\mathcal F_{k+n}^{\infty}$ are
independent for all $n>m$, so $\varphi(n)=\alpha(n)=0$ for $n>m$. The
mixing setting is thus a quantitative relaxation of $m$-dependence, and
every result below reduces to its $m$-dependent counterpart in
\cite{Panraksa2026} on substituting $\varphi(n)=0$ for $n>m$, where the
relevant parameter ranges match (the second-order recovery requires
$m\ge2$).
\end{remark}

We record three elementary tools.

\begin{lemma}[Product-to-exponential bound]\label{lem:prodexp}
For any integer $M\ge1$ and $x_{1},\dots,x_{M}\in[0,1]$,
$\prod_{j=1}^{M}(1-x_{j})\le\exp\bigl(-\sum_{j=1}^{M}x_{j}\bigr)$.
\end{lemma}

\begin{proof}
If some $x_{j}=1$ the product is zero and the claim is immediate;
otherwise $x_{j}\in[0,1)$ for every $j$, and applying $\log(1-x)\le-x$
term by term and exponentiating gives the result.
\end{proof}

\begin{lemma}[Approximate independence under $\varphi$-mixing]\label{lem:phiapprox}
Let $\mathcal A_{1},\dots,\mathcal A_{M}$ be sub-$\sigma$-algebras of
$\mathcal F$ satisfying
$\varphi\bigl(\sigma(\mathcal A_{1}\cup\cdots\cup\mathcal A_{r}),
\mathcal A_{r+1}\bigr)\le\varphi_{\star}$ for $1\le r\le M-1$. Let
$C_{j}\in\mathcal A_{j}$ and $u_{j}:=\P(C_{j})$. Then
\begin{equation}\label{eq:phiapprox}
  \P\Bigl(\,\bigcap_{j=1}^{M}C_{j}^{\mathrm c}\,\Bigr)
  \;\le\;\prod_{j=1}^{M}\min\bigl(1,\,1-u_{j}+\varphi_{\star}\bigr)
  \;\le\;\exp\!\Bigl(-\sum_{j=1}^{M}(u_{j}-\varphi_{\star})_{+}\Bigr).
\end{equation}
\end{lemma}

\begin{proof}
We prove the left inequality by induction on $M$; the case $M=1$ is
$\P(C_{1}^{\mathrm c})=1-u_{1}\le\min(1,1-u_{1}+\varphi_{\star})$. Write
$D_{r}:=\bigcap_{j=1}^{r}C_{j}^{\mathrm c}\in
\sigma(\mathcal A_{1}\cup\cdots\cup\mathcal A_{r})$. If
$\P(D_{M-1})=0$ the bound is immediate; otherwise, by the definition of
$\varphi$,
$|\P(C_{M}^{\mathrm c}\mid D_{M-1})-\P(C_{M}^{\mathrm c})|\le
\varphi_{\star}$, so
$\P(C_{M}^{\mathrm c}\mid D_{M-1})\le\min(1,1-u_{M}+\varphi_{\star})$.
Multiplying by $\P(D_{M-1})$ and using the induction hypothesis gives the
product bound. For the exponential bound,
$\min(1,1-u+\varphi_{\star})\le\exp(-(u-\varphi_{\star})_{+})$: the
right side is $1$ if $u\le\varphi_{\star}$, while if $u>\varphi_{\star}$
then $1-u+\varphi_{\star}\le e^{-(u-\varphi_{\star})}$ by $1-x\le e^{-x}$.
\end{proof}

\begin{lemma}[Event-level strong-mixing bound]\label{lem:alpha}
If $\alpha(\mathcal A,\mathcal B)\le\alpha_{\star}$, then
$|\P(A\cap B)-\P(A)\P(B)|\le\alpha_{\star}$ for all $A\in\mathcal A$,
$B\in\mathcal B$; equivalently
$|\Cov(\mathbf 1_{A},\mathbf 1_{B})|\le\alpha_{\star}$.
\end{lemma}

\begin{proof}
This is exactly the defining supremum in
Definition~\ref{def:coeff}; the covariance form follows by writing
\(\Cov(\mathbf 1_A,\mathbf 1_B)=\P(A\cap B)-\P(A)\P(B)\).
\end{proof}

\noindent\emph{Blocking construction.}
For $N\ge1$, $L\ge0$ and $r\in\{1,\dots,L+1\}$ let
\[
  J_{r}(L):=\{\,k\in\{1,\dots,N\}:k\equiv r\!\!\pmod{L+1}\,\}.
\]
The classes $J_{1}(L),\dots,J_{L+1}(L)$ partition $\{1,\dots,N\}$, each
consisting of indices at least $L+1$ apart, and
$\sum_{r=1}^{L+1}\sum_{k\in J_{r}(L)}w_{k}=\sum_{k=1}^{N}w_{k}$ for any
weights $w_{k}\ge0$, so
$\max_{r}\sum_{k\in J_{r}(L)}w_{k}\ge\frac{1}{L+1}\sum_{k=1}^{N}w_{k}$.
Enumerating a class as $k_{r,1}<k_{r,2}<\cdots$, the gaps are $\ge L+1$,
so the past block $\sigma(A_{k_{r,1}},\dots,A_{k_{r,j-1}})\subseteq
\mathcal F_{1}^{k_{r,j-1}}$ and the next event
$\sigma(A_{k_{r,j}})\subseteq\mathcal F_{k_{r,j-1}+L+1}^{\infty}$ (and
$\subseteq\mathcal F_{k_{r,j-1}+L+1}^{N}$ in the finite restricted
reading); by monotonicity the coefficient between them is at most
$\varphi(L+1)$ (respectively $\alpha(L+1)$). Taking $L=m$ in the
$m$-dependent case recovers the residue-class split of
\cite{Panraksa2026}.

\section{Main results}\label{sec:main}

\subsection{A finite-sample inequality under $\varphi$-mixing}

\begin{theorem}[One-lag $\varphi$-coefficient inequality]\label{thm:phi}
Let $N\ge1$ and let $(A_{k})_{1\le k\le N}$ be events, with
$\varphi(\cdot)$ fixed by one coefficient convention of
Section~\ref{sec:prelim}. For each integer $L\ge0$,
\begin{equation}\label{eq:phi-main}
  \P\Bigl(\,\bigcup_{k=1}^{N}A_{k}\,\Bigr)
  \;\ge\; 1-\exp\!\Biggl(-\frac{1}{L+1}
  \sum_{k=1}^{N}\bigl(\P(A_{k})-\varphi(L+1)\bigr)_{+}\Biggr).
\end{equation}
In particular, if $\varphi(L+1)\le\min_{1\le k\le N}\P(A_{k})$,
\begin{equation}\label{eq:phi-clean}
  \P\Bigl(\,\bigcup_{k=1}^{N}A_{k}\,\Bigr)
  \;\ge\; 1-\exp\!\Bigl(-\tfrac{1}{L+1}\bigl(S_{N}-N\varphi(L+1)\bigr)\Bigr).
\end{equation}
\end{theorem}

\begin{proof}
Fix $L\ge0$ and recall the classes $J_{r}(L)$; empty classes are ignored
(their intersection is $\Omega$). For a non-empty class enumerate
$k_{r,1}<\cdots<k_{r,n_{r}}$ and set
$\mathcal A_{j}:=\sigma(A_{k_{r,j}})$. By the blocking construction the
hypothesis of Lemma~\ref{lem:phiapprox} holds with
$\varphi_{\star}=\varphi(L+1)$, so with $C_{j}=A_{k_{r,j}}$,
\[
  \P\Bigl(\,\bigcap_{k\in J_{r}(L)}A_{k}^{\mathrm c}\,\Bigr)
  \;\le\;\exp\!\Bigl(-\!\!\sum_{k\in J_{r}(L)}\!\!
  \bigl(\P(A_{k})-\varphi_{\star}\bigr)_{+}\Bigr).
\]
Since $\bigcap_{k=1}^{N}A_{k}^{\mathrm c}\subseteq
\bigcap_{k\in J_{r}(L)}A_{k}^{\mathrm c}$ for each $r$, and the classes
partition $\{1,\dots,N\}$,
\[
  \sum_{r=1}^{L+1}\sum_{k\in J_{r}(L)}\bigl(\P(A_{k})-\varphi_{\star}\bigr)_{+}
  =\sum_{k=1}^{N}\bigl(\P(A_{k})-\varphi_{\star}\bigr)_{+},
\]
so some class $r^{*}$ has
$\sum_{k\in J_{r^{*}}(L)}(\P(A_{k})-\varphi_{\star})_{+}\ge
\frac{1}{L+1}\sum_{k=1}^{N}(\P(A_{k})-\varphi_{\star})_{+}$. Using this
class,
\[
  \P\Bigl(\,\bigcap_{k=1}^{N}A_{k}^{\mathrm c}\,\Bigr)
  \;\le\;\exp\!\Biggl(-\frac{1}{L+1}
  \sum_{k=1}^{N}\bigl(\P(A_{k})-\varphi_{\star}\bigr)_{+}\Biggr),
\]
which is \eqref{eq:phi-main} after taking complements. If
$\varphi_{\star}\le\min_{k}\P(A_{k})$ then each summand equals
$\P(A_{k})-\varphi_{\star}$, giving \eqref{eq:phi-clean}.
\end{proof}

When $(A_{k})$ is $m$-dependent, $L=m$ gives $\varphi(m+1)=0$ and
\eqref{eq:phi-main} reduces to \eqref{eq:mdep-baseline}, i.e.\
\cite[Theorem~2]{Panraksa2026}. Since $L$ is free, it may be optimised
\emph{a posteriori}.

\begin{corollary}[Optimised exponent]\label{cor:opt}
Under the hypotheses of Theorem~\ref{thm:phi},
$\P\bigl(\bigcup_{k=1}^{N}A_{k}\bigr)\ge1-\exp(-\Psi)$, where
\[
  \Psi:=\sup_{L\in\mathbb Z_{\ge0}}\frac{1}{L+1}
  \sum_{k=1}^{N}\bigl(\P(A_{k})-\varphi(L+1)\bigr)_{+}.
\]
In the finite restricted reading the supremum may be taken over
$0\le L\le N-1$.
\end{corollary}

\begin{proof}
Each $L$ gives \eqref{eq:phi-main}; as $x\mapsto1-e^{-x}$ is increasing
on $[0,\infty)$ and the exponents are nonnegative, take the supremum
over $L$. For $L\ge N$ the restricted coefficient at lag $L+1$ is $0$ and
the exponent $S_{N}/(L+1)$ does not exceed its value at $L=N-1$.
\end{proof}

\subsection{An $\alpha$-mixing analogue}

For $\alpha$-mixing sequences Lemma~\ref{lem:phiapprox} is unavailable
and is replaced by the covariance bound of Lemma~\ref{lem:alpha}, which
accumulates additively.

\begin{theorem}[One-lag $\alpha$-coefficient inequality]\label{thm:alpha}
Let $N\ge1$ and let $(A_{k})_{1\le k\le N}$ be events, with
$\alpha(\cdot)$ fixed by one coefficient convention of
Section~\ref{sec:prelim}. For each integer $L\ge0$,
\begin{equation}\label{eq:alpha-main}
  \P\Bigl(\,\bigcup_{k=1}^{N}A_{k}\,\Bigr)
  \;\ge\;\Bigl[\,1-\exp\!\bigl(-\tfrac{S_{N}}{L+1}\bigr)
  -\bigl\lceil\tfrac{N}{L+1}\bigr\rceil\alpha(L+1)\,\Bigr]_{+}.
\end{equation}
\end{theorem}

\begin{proof}
Fix $L\ge0$ and the classes $J_{r}(L)$; ignore empty classes. For a
non-empty class enumerate $k_{r,1}<\cdots<k_{r,n_{r}}$, write
$p_{r,j}:=\P(A_{k_{r,j}})$, $q_{0}:=1$, and
$q_{j}:=\P\bigl(\bigcap_{i\le j}A_{k_{r,i}}^{\mathrm c}\bigr)$. For
$j\ge2$, with $D_{j-1}:=\bigcap_{i<j}A_{k_{r,i}}^{\mathrm c}\in
\mathcal F_{1}^{k_{r,j-1}}$ and
$A_{k_{r,j}}^{\mathrm c}\in\mathcal F_{k_{r,j-1}+L+1}^{\infty}$, the
blocking construction and Lemma~\ref{lem:alpha} give
$|\P(D_{j-1}\cap A_{k_{r,j}}^{\mathrm c})-\P(D_{j-1})\P(A_{k_{r,j}}^{\mathrm c})|
\le\alpha(L+1)=:\alpha_{\star}$, hence
$q_{j}\le(1-p_{r,j})q_{j-1}+\alpha_{\star}$. Induction on $j$ yields
\[
  q_{n_{r}}\;\le\;\prod_{j=1}^{n_{r}}(1-p_{r,j})
  +\alpha_{\star}\sum_{j=2}^{n_{r}}\prod_{i=j+1}^{n_{r}}(1-p_{r,i})
  \;\le\;\exp\!\Bigl(-\!\!\sum_{k\in J_{r}(L)}\!\!\P(A_{k})\Bigr)
  +(n_{r}-1)_{+}\alpha_{\star},
\]
using Lemma~\ref{lem:prodexp} and bounding each inner product by $1$.
Choosing a class with $\sum_{k\in J_{r^{*}}(L)}\P(A_{k})\ge S_{N}/(L+1)$
and $n_{r^{*}}\le\lceil N/(L+1)\rceil$, then taking complements and using
nonnegativity of the left side gives \eqref{eq:alpha-main}.
\end{proof}

The bound \eqref{eq:alpha-main} is informative when
$\lceil N/(L+1)\rceil\alpha(L+1)$ is small relative to
$1-\exp(-S_{N}/(L+1))$.

\subsection{A second-order Bonferroni refinement}

When local pairwise overlaps are available, a shifted-block argument
brings them into the exponent. Set
$T_{L-1}:=\sum_{1\le i<j\le N,\ |i-j|\le L-1}\P(A_{i}\cap A_{j})$.

\begin{theorem}[Second-order one-lag $\varphi$-coefficient inequality]\label{thm:second}
Under the hypotheses of Theorem~\ref{thm:phi} with $L\ge2$,
\begin{equation}\label{eq:second}
  \P\Bigl(\,\bigcup_{k=1}^{N}A_{k}\,\Bigr)
  \;\ge\; 1-\exp\!\Bigl(-\tfrac12\bigl(S_{N}-T_{L-1}
  -\kappa_{N,L}\,\varphi(L+1)\bigr)_{+}\Bigr),
  \qquad \kappa_{N,L}:=\bigl\lceil\tfrac{N}{L}\bigr\rceil+1.
\end{equation}
\end{theorem}

\begin{proof}
For each shift $r\in\{0,1,\dots,L-1\}$, partition $\{1,\dots,N\}$ into
consecutive length-$L$ blocks
$I_{j}^{(r)}:=\{r+(j-1)L+1,\dots,r+jL\}\cap\{1,\dots,N\}$ ($j\ge0$), and
set $B_{j}^{(r)}:=\bigcup_{k\in I_{j}^{(r)}}A_{k}$; empty blocks are
ignored. Consecutive same-parity blocks $I_{j}^{(r)},I_{j+2}^{(r)}$ are
separated by at least $L+1$ indices, and truncation to $\{1,\dots,N\}$
does not decrease this separation. Hence, ordering the non-empty even-
and the non-empty odd-indexed block-union events, the past of each lies
in some $\mathcal F_{1}^{t}$ and the next same-parity block in
$\mathcal F_{t+L+1}^{\infty}$ (or in $\mathcal F_{t+L+1}^{N}$ in the
finite restricted reading); by the blocking construction the one-sided
coefficient is at most $\varphi(L+1)$. Let $\mathcal N_{r}$ be the number
of non-empty blocks, so $\mathcal N_{r}\le\lceil N/L\rceil+1=\kappa_{N,L}$.

\emph{Parity split.} Applying Lemma~\ref{lem:phiapprox} to the even- and
to the odd-indexed block-union events and combining the two bounds via
$\min\{e^{-x},e^{-y}\}\le e^{-(x+y)/2}$ gives, for every $r$,
\begin{equation}\label{eq:parity}
  \P\Bigl(\,\bigcap_{k=1}^{N}A_{k}^{\mathrm c}\,\Bigr)
  \;\le\;\exp\!\Bigl(-\tfrac12\sum_{j\ge0}
  \bigl(\P(B_{j}^{(r)})-\varphi(L+1)\bigr)_{+}\Bigr).
\end{equation}

\emph{Bonferroni and pair counting.} The second-order Bonferroni
inequality gives
$\P(B_{j}^{(r)})\ge\sum_{k\in I_{j}^{(r)}}\P(A_{k})-
\sum_{i<k:\,i,k\in I_{j}^{(r)}}\P(A_{i}\cap A_{k})$. A pair $i<k$ with
$d:=k-i$ lies in a common length-$L$ block for exactly $(L-d)_{+}$ of the
$L$ shifts: if $d<L$, the $d$ internal boundaries between $i$ and
$k$ have distinct residues modulo $L$; if $d\ge L$, every residue occurs
among those boundary residues, so no shift places the pair in a common
block. Summing over $j$ and averaging over $r$,
\begin{equation}\label{eq:blocklower}
  \frac1L\sum_{r=0}^{L-1}\sum_{j\ge0}\P(B_{j}^{(r)})
  \;\ge\; S_{N}-T_{L-1}.
\end{equation}

\emph{Positive part and selection.} Using $(a-c)_{+}\ge a-c$ and
$\mathcal N_{r}\le\kappa_{N,L}$, and then applying \eqref{eq:blocklower},
\[
  \frac1L\sum_{r=0}^{L-1}\sum_{j\ge0}
  \bigl(\P(B_{j}^{(r)})-\varphi(L+1)\bigr)_{+}
  \;\ge\;(S_{N}-T_{L-1})-\kappa_{N,L}\,\varphi(L+1).
\]
The union $\bigcup_{k}A_{k}$ is the same for every shift, so the bound
\eqref{eq:parity} may be minimised over $r$; with
$x_{r}:=\sum_{j\ge0}(\P(B_{j}^{(r)})-\varphi(L+1))_{+}\ge0$, one has
$\min_{r}e^{-x_{r}/2}\le\exp(-\tfrac1{2L}\sum_{r}x_{r})$. Combining the
last three displays and inserting the positive part (legitimate since the
left side is non-negative) yields \eqref{eq:second}.
\end{proof}

For $m\ge2$, taking $L=m$ in the $m$-dependent case gives
$\varphi(m+1)=0$, so \eqref{eq:second} reduces to the positive-part
exponent $\tfrac12(S_{N}-T_{m-1})_{+}$, the form of
\cite[Theorem~4]{Panraksa2026}.

\subsection{Sharpness of the spacing constant}

The spacing constant inherited from the $m$-dependent case cannot be
improved uniformly over the mixing classes, because those classes
contain finite-range dependent examples with vanishing coefficient at
the relevant lag.

\begin{proposition}[Sharpness of the $m$-dependent constant]\label{prop:sharp}
For every integer $m\ge0$, let $c_{m}^{*}$ be the supremum of all $c>0$
such that $\P\bigl(\bigcup_{k=1}^{N}A_{k}\bigr)\ge1-\exp(-c\,S_{N})$ for
every $N$, every probability space, and every finite $m$-dependent
family. Then $c_{m}^{*}=1/(m+1)$.
\end{proposition}

\begin{proof}
$c_{m}^{*}\ge1/(m+1)$ follows from Theorem~\ref{thm:phi} at $L=m$, since
$\varphi(m+1)=0$ (Remark~\ref{rem:mdep}). For the reverse bound, fix
$p\in(0,1)$ and a positive integer $q$, let $B_{1},\dots,B_{q}$ be
independent with $\P(B_{j})=p$, and set
$A_{(j-1)(m+1)+r}:=B_{j}$ for $r=1,\dots,m+1$, $N:=(m+1)q$. Index sets
separated by more than $m$ meet disjoint blocks, hence depend on disjoint
subfamilies of $B_{1},\dots,B_{q}$, so $(A_{k})$ is $m$-dependent. Here
$S_{N}=(m+1)qp$ and
$\P(\bigcup_{k}A_{k})=\P(\bigcup_{j}B_{j})=1-(1-p)^{q}$. Admissibility of
$c$ forces $1-(1-p)^{q}\ge1-\exp(-c(m+1)qp)$, i.e.\
$-\log(1-p)\ge c(m+1)p$ for all $p\in(0,1)$. Since
$\inf_{0<p<1}\frac{-\log(1-p)}{p}=\lim_{p\downarrow0}\frac{-\log(1-p)}{p}=1$,
this gives $c\le1/(m+1)$.
\end{proof}

\begin{corollary}[Sharpness under mixing]\label{cor:sharpmix}
For every integer $L\ge0$, the constant $1/(L+1)$ in the exponent of
\eqref{eq:phi-main} and in that of \eqref{eq:alpha-main} cannot be
replaced by any strictly larger universal spacing constant in a bound of
the same zero-residual form -- even restricted to the subclass of
$\varphi$-mixing (respectively $\alpha$-mixing) sequences with
$\varphi(L+1)=0$ (respectively $\alpha(L+1)=0$).
\end{corollary}

\begin{proof}
Specialising Theorems~\ref{thm:phi} and~\ref{thm:alpha} to sequences for
which the coefficient vanishes at lag $L+1$ gives
$\P(\bigcup_{k}A_{k})\ge1-\exp(-S_{N}/(L+1))$. Under the finite restricted
convention the families of Proposition~\ref{prop:sharp} (with $m=L$)
realise this class directly; for the ambient convention, embed that
finite block construction into an infinite sequence of independent blocks
of length $L+1$, which is $L$-dependent and hence has
$\varphi(L+1)=\alpha(L+1)=0$. Thus the constant cannot be enlarged.
\end{proof}

This sharpness is a zero-residual statement: it does not assert
optimality of the residual penalties $\varphi(L+1)$, $\alpha(L+1)$. For
genuinely mixing sequences with positive residual at every lag,
Theorem~\ref{thm:phi} presents a trade-off -- decreasing $L$ improves
the prefactor $1/(L+1)$ but increases $\varphi(L+1)$ -- captured by the
quantity $\Psi$ of Corollary~\ref{cor:opt}.

\section{An example and comparisons}\label{sec:ex}

\noindent\emph{Recovery of the $m$-dependent case.}
If $(A_{k})$ is $m$-dependent then $\varphi(n)=0$ for $n>m$, and $L=m$
turns \eqref{eq:phi-main} into \eqref{eq:mdep-baseline}, recovering
\cite[Theorem~2]{Panraksa2026}; for $m\ge2$ it turns \eqref{eq:second}
into the positive-part bound $1-\exp(-\tfrac12(S_{N}-T_{m-1})_{+})$ of
\cite[Theorem~4]{Panraksa2026}. Thus no information is lost as
$\varphi(L+1)\downarrow0$.

\medskip
\noindent\emph{Geometrically $\varphi$-mixing sequences.}
Suppose $\varphi(n)\le C\rho^{n}$ for constants $C\ge1$, $\rho\in(0,1)$;
this holds, for instance, under a Doeblin minorisation condition, and in
the finite-state case whenever the chain is irreducible and aperiodic
\cite{Doukhan1994,Bradley2005}.

\begin{proposition}[Geometric $\varphi$-mixing with positive lower mass]\label{prop:geom}
Suppose $\varphi(n)\le C\rho^{n}$ and $p:=\inf_{k}\P(A_{k})>0$. Choose an
integer $L_{0}\ge0$ with $C\rho^{L_{0}+1}\le p/2$. Then for every
$N\ge1$,
\[
  \P\Bigl(\,\bigcup_{k=1}^{N}A_{k}\,\Bigr)
  \;\ge\; 1-\exp\!\Bigl(-\frac{S_{N}}{2(L_{0}+1)}\Bigr)
  \;\ge\; 1-\exp\!\Bigl(-\frac{pN}{2(L_{0}+1)}\Bigr).
\]
\end{proposition}

\begin{proof}
By the choice of $L_{0}$,
$\varphi(L_{0}+1)\le p/2\le\min_{k}\P(A_{k})$, so \eqref{eq:phi-clean}
with $L=L_{0}$ gives
$\P(\bigcup_{k}A_{k})\ge1-\exp\bigl(-(S_{N}-N\varphi(L_{0}+1))/(L_{0}+1)\bigr)$.
Since $N\varphi(L_{0}+1)\le Np/2\le S_{N}/2$, the first bound follows,
and the second from $S_{N}\ge pN$.
\end{proof}

For a concrete instance, let $(X_{k})$ be the stationary Markov chain on
$\{0,1\}$ with transition matrix
$P=\begin{pmatrix}1-a & a\\ b & 1-b\end{pmatrix}$,
$a,b\in(0,1)$, and let $A_{k}=\{X_{k}=1\}$, so $\P(A_{k})=a/(a+b)$ and the
event coefficients satisfy $\varphi_{A}(n)\le|1-a-b|^{n}$: indeed, with
$\pi_{1}=a/(a+b)$ and $\pi_{0}=b/(a+b)$,
$P^{n}(0,1)-\pi_{1}=-\pi_{1}(1-a-b)^{n}$ and
$P^{n}(1,1)-\pi_{1}=\pi_{0}(1-a-b)^{n}$, so the event-level coefficient is
at most $\max(\pi_{0},\pi_{1})\,|1-a-b|^{n}\le|1-a-b|^{n}$. If $a+b=1$, the
chain is independent at positive lags and Proposition~\ref{prop:geom}
applies with zero residual; otherwise, set $\rho=|1-a-b|\in(0,1)$, $C=1$,
and any fixed $L_{0}$ with $\rho^{L_{0}+1}\le a/(2(a+b))$, which gives a
positive exponent linear in $N$ with no logarithmic loss. The exact
stationary probability is $1-\frac{b}{a+b}(1-a)^{N-1}$ for comparison.

\medskip
\noindent\emph{Comparison with classical bounds.}
For finite $N$, the natural comparator is the Chung--Erd\H{o}s
second-moment bound
$\P(\bigcup_{k}A_{k})\ge S_{N}^{2}/\sum_{i,j}\P(A_{i}\cap A_{j})$
\cite{ChungErdos1952}, and its variance form
$1-\Var(Z_{N})/S_{N}^{2}$ with $Z_{N}=\sum_{k}\mathbf 1_{A_{k}}$, which
underlies the asymptotic criteria of Erd\H{o}s--R\'enyi
\cite{ErdosRenyi1959} and the weak-dependence criteria of
\cite{DedeckerMerlevedeRio2022}. These are often sharper when either the full
matrix of pairwise intersections or a sharp variance estimate is
available. The bounds \eqref{eq:phi-main} and \eqref{eq:alpha-main} are
complementary: they require only the marginals and a mixing coefficient
at a single chosen spacing, and the second-order bound
\eqref{eq:second} sits between the two, using local overlaps together
with a one-lag residual. They are therefore one-lag blocking tools, not
replacements for pairwise-information \cite{Frolov2012} or full-profile
variance methods. Classical second-moment and asymptotic refinements --
the Kochen--Stone theorem \cite{KochenStone1964} and the treatments in
Chandra's monograph \cite{Chandra2012} -- instead control limsup events
or finite unions through global pairwise information; the present
estimates differ in using only one-dimensional marginals and a single
selected-lag mixing coefficient.

\medskip
\noindent\emph{Concluding remarks.}
The results form a small finite-sample toolkit attached to a single
mixing coefficient value: a $\varphi$-bound with the residual inside the
exponent, an $\alpha$-bound with an additive residual, a second-order
refinement, and a matching sharpness statement, all reducing to the
$m$-dependent inequalities of \cite{Panraksa2026} at zero residual.

\section*{Acknowledgements}
The author thanks the anonymous referee for a careful reading and helpful
comments, and Pakawut Jiradilok for the suggestion on the optimality of the
$m$-dependent constant.


\end{document}